\documentclass[reqno]{amsart}

\usepackage{amsmath,amssymb,amsthm,comment,enumerate,mathrsfs}
\usepackage{hyperref}
\usepackage{color}
\usepackage{soul}

\swapnumbers

\numberwithin{equation}{section}

\newtheorem{Thm}{Theorem}[section]
\newtheorem{Prop}[Thm]{Proposition}
\newtheorem{Lem}[Thm]{Lemma}

\theoremstyle{definition}
\newtheorem{Rem}[Thm]{Remark}

\newtheorem{Expl}[Thm]{Example}

\newcommand{\R}{\mathbb{R}}

\newcommand{\E}[1]{{\rm E}(#1)}

\newcommand{\e}{{\rm e}}

\newcommand{\eps}{\varepsilon}
\newcommand{\es}{\emptyset}

\newcommand{\olReal}{\overline{\R}}
\renewcommand{\rho}{\varrho}

\newcommand{\sm}{\setminus}

\title{Injective Subsets of $l_{\infty}(I)$}

\author{
Dominic Descombes and
Ma\"el Pav\' on}
\address{Department of Mathematics, ETH Z\"urich, 8092 Z\"urich, Switzerland}
\email{dominic.descombes@math.ethz.ch}
\email{mael.pavon@math.ethz.ch}

\date{\today}

\begin{document}

\maketitle
\begin{abstract}
We give an explicit characterization of all injective subsets of the model space $l_{\infty}(I)$ for a general set $I$, in terms of inequalities involving $1$-Lipschitz functions. Since the class of all injective metric spaces coincides with the one of all  absolute $1$-Lipschitz retracts, the present work yields a characterization of all the subsets of $l_{\infty}(I)$ that are absolute $1$-Lipschitz retracts.
\end{abstract}


\section{Introduction}

A metric space $(X,d)$ is said to be {\em injective} if for every isometric embedding $i \colon A \to B$ of metric spaces $(A,d_A),(B,d_B)$ and every $1$-Lipschitz map $f \colon A \to X$, there exists a $1$-Lipschitz map $\overline{\!f}$ such that $f = \overline{\!f} \circ i$. Basic examples of injective metric spaces include $\R$, the space $l_{\infty}(I)$ for any set $I$, as well as $\R$-trees. By $l_{\infty}(I)$ is meant the space of all real-valued bounded functions $f$ endowed with the norm $\| f \|_\infty:=\sup_{i \in I} |f_i|$ which we denote by $\| f \|$ for notational convenience. In the context of subsets of $l_{\infty}(I)$, we need to introduce some pieces of notation. Let moreover $\widehat{\pi}_i \colon l_{\infty}(I) \to  l_{\infty}(I \setminus \{ i \})$ be the map given by dropping the $i$-th coordinate. 

This work proves a metric characterization of the injective subsets of $l_{\infty}(I)$ in terms of systems of inequalities given by $1$-Lipschitz functions. Namely,

\begin{Thm}\label{Thm:ISLI-CharacterizationTheorem}
A non-empty subset $Q$ of $l_{\infty}(I)$ is injective if and only if it can be written as
\begin{equation}\label{eq:ISLI-ExpressionCharacterization}
Q = \{ x\in l_{\infty}(I) :  (\underline{r}_i \circ \widehat{\pi}_i) \leq x_i \leq  (\overline{r}_i \circ \widehat{\pi}_i) \text{ for all } i\in I \} 
\end{equation}
where $\underline{r\!}\,_i, \bar{r}_i : l_{\infty}(I \setminus \{i\}) \rightarrow \R$ are $1$-Lipschitz maps satisfying $\underline{r\!}\,_i \leq\bar{r}_i $, that is $\underline{r}_i(y) \le \overline{r}_i(y)$ for all $y \in l_{\infty}(I \setminus \{i\})$, possibly dropping a subcollection of the collection of all inequalities appearing in~\eqref{eq:ISLI-ExpressionCharacterization}.
\end{Thm}

The characterization is explicit, namely it provides a concrete expression for each single injective subset of each of the model spaces $l_{\infty}(I)$. The proof of this characterization is based on the proof by the first author of the characterization in the particular case where $I=\{1,\dots,n\}$, that is in the particular case where $l_{\infty}(I)$ corresponds to $\R^n$ endowed with the maximum norm cf. \cite{Des}. In the next section, we prove among others the equivalence between hyperconvexity and the absolute $1$-Lipschitz retract property.

\section{Preliminaries on Absolute 1-Lipschitz Retracts, Hyperconvexity and Isbell's Injective Hull}

Let us start by recalling two characterizations of injective metric spaces. A metric space $(X,d)$ is called an \emph{absolute $1$-Lipschitz retract (or $1$-ALR)} if for every isometric embedding $i \colon X \to Y$ into a metric space $Y$, there exists a $1$-Lipschitz retraction of $Y$ onto $i(X)$. To show that this property is equivalent to injectivity, assume first that $X$ is injective. If now $i \colon X \to Y$ is an isometric embedding, $i(X)$ is injective. Thus, the identity map on $i(X)$ extends to a $1$-Lipschitz retraction $\rho \colon Y \to i(X)$ showing that $X$ is a $1$-ALR. Conversely, every metric space $X$ embeds isometrically into $l_\infty(X)$ via a Kuratowski embedding $k_{x_0} \colon x \mapsto d_x - d_{x_0}$ for an arbitrarily chosen base point $x_0 \in X$. Hence, if $X$ is a $1$-ALR, then $k_{x_0}(X)$ is a $1$-Lipschitz retract in $l_\infty(X)$. Therefore, by injectivity of $l_\infty(X)$ it follows that $X$ is injective as well.

Another characterization of injective metric spaces relies upon an intersection property of metric balls. In a metric space $(X,d)$, we use the notation $B(x,r):= \{ y \in X : d(x,y) \le r\}$. One says that $(X,d)$ is \emph{hyperconvex} if every family $\{(x_\gamma,r_\gamma)\}_{\gamma \in \Gamma}$ in $X \times \R$  satisfying $r_\beta + r_\gamma \ge d(x_\beta,x_\gamma)$ for all pairs of indices $\beta,\gamma \in \Gamma$, has the property that $\bigcap_{\gamma \in \Gamma}B(x_\gamma,r_\gamma) \ne \es$. (As a matter of convention, the intersection equals $X$ if $\Gamma = \es$, so that hyperconvex spaces are non-empty by definition.) To see that injectivity implies hyperconvexity, let $\{(x_\gamma,r_\gamma)\}_{\gamma \in \Gamma} \in X \times \R$ be a family with $r_\beta + r_\gamma \ge d(x_\beta,x_\gamma)$ for all $\beta, \gamma \in \Gamma$. Let $A:= \{x_\gamma\}_{\gamma \in \Gamma}$ be endowed with the metric $d_A$ induced by $(X,d)$. Set $B := A \cup \{ b\}$ where $d_B(x_\gamma,b) := r_\gamma$. By our assumptions, $d_B$ defines a metric on $B$. By injectivity of $X$, there is a $1$-Lipschitz map $\bar{j} \colon B \to X$ such that the inclusions $j \colon A \to X$ and $i \colon A \to B$ satisfy $\bar{j} \circ i = j$. It follows that $\bar{j}(b) \in \bigcap_{\gamma \in \Gamma}B(x_\gamma,r_\gamma)$, which in turn shows that $X$ is hyperconvex. For the proof of the converse, note that if $f \colon A \to X $ is $1$-Lipschitz, $i \colon A \to B$ is an isometric embedding, and $b \in B \sm i(A)$, then $d_B(i(a),b) + d_B(i(a'),b) \ge d_A(a,a') \ge d(f(a),f(a'))$ for all $a,a' \in A$. Hence, if $X$ is hyperconvex, then $S:=\bigcap_{a \in A}B(f(a),d_B(i(a),b))$ is non-empty, and one obtains a $1$-Lipschitz map $f_b \colon i(A) \cup \{b\} \to X$ by setting $f_b(i(a)) := f(a)$ on $A$ and taking $f_b(b)$ to be any point in $S$. Using Zorn's lemma one can prove the existence of a map $\bar{f}$ with the desired properties, from which one deduces that $X$ is injective. A direct consequence of this characterization is that the intersection of any family of closed balls in an injective metric space is itself injective if and only if it is non-empty.

Isbell proved that every metric space $(X,d)$ possesses an {\em injective hull} $(i,Y)$ (later, when considering Isbell's injective hull $(\e, \mathrm{E}(X))$ we usually write $\mathrm{E}(X)$ for simplicity) which is unique up to isometry and minimal among injective spaces containing an isometric copy of $X$. This means that $i \colon X \to Y$ is an isometric embedding with the following property: whenever there is a metric space $Z$ and a $1$-Lipschitz map $h \colon Y \to Z$ so that $h\circ i$ is an isometric embedding, it follows that $h$ is an isometric embedding as well. 

We now need to give a short outline of some elements of Isbell's construction for later use. For a more comprehensive introduction to injective spaces and the construction of $\E{X}$, see for instance \cite{Lan}. Given a metric space $(X,d)$, we denote by $\R^X$ the vector space of real-valued functions defined on $X$ and we set
\begin{equation}\label{eq:ISLI-DeltaX}
\Delta(X) := \{ f \in \R^X : \text{$f(x) + f(y) \ge d(x,y)$ for all $x,y \in X$}\}.
\end{equation}
For $f,g \in \R^X$, the inequality $g\leq f$ means that $g(x)\leq f(x)$ for all $x$.
A function $f \in \Delta(X)$ is called {\em extremal\/} if there is no $g \le f$ in $\Delta(X)$ 
different from $f$. One can show that the collection $\E X $ of all extremal functions in $\Delta(X)$ is equivalently given by
\begin{equation}\label{eq:ISLI-EX}
\E X = \bigl\{ f \in \R^X : f(x) = \textstyle \sup_{y \in X}(d(x,y) - f(y)) \text{ for all } x \in X \bigr\}.
\end{equation}
Thus, $f \in \E X$ if and only if $f \in \Delta(X)$ and for every $x \in X$ and $\eps>0$, there is an $y \in X$ so that
\begin{equation}\label{eq:extremal}
f(x)+f(y)\leq d(x,y)+\eps .
\end{equation}
Applying the equation defining the members of $\E X $ twice, we obtain that for $f \in \E X$ and for all $x,x'\in X$, one has
\[
f(x)-d(x,x')=\sup_{y\in X}(d(x,y)-d(x,x')-f(y))\leq f(x').
\]
This implies that extremal functions are $1$-Lipschitz. One can also prove that the map $(f,g) \mapsto \sup_{x \in X}|f(x) - g(x)|$ endows $\E X$ with a metric. It follows from \eqref{eq:ISLI-DeltaX} that every function in $\Delta(X)$ is non-negative. It is easy to see that for each $x\in X$, the function $d_x$ given by the assignement $y \mapsto d(x,y)$ is extremal, and the functions of this form are the only extremal functions with zeros. One can show that $(\e,\E X)$ is in fact an injective hull of $X$. Thus, to \textit{every} metric space $(X,d)$ corresponds Isbell's injective hull $(\e, \E X)$ where $\E X$ isometrically embeds into $l_{\infty}(X)$ via a Kuratowski embedding.

\section{The Characterization}

\begin{Thm}[\cite{Bai,EspK}]\label{Thm:IS-TheoremBaillon}
If $\Theta$ is a totally ordered set and $(H_{\theta})_{\theta \in \Theta}$ is a decreasing family of nonempty bounded hyperconvex spaces, then the intersection is hyperconvex.
\end{Thm}

Recall that for a metric space $(X,d)$, we write $B(x,r):= \{ y \in X : d(x,y) \le r\}$. A non-empty subset $A$ of a metric space $(X,d)$ is \emph{externally hyperconvex} in $X$ if for any collection $\{(x_\gamma,r_\gamma)\}_{\gamma \in \Gamma} \subset X \times [0,\infty)$ satisfying $d(x_\beta,x_\gamma) \leq r_\beta+r_\gamma$ and $d(x_\gamma,A)\leq r_\gamma$ for all $\beta, \gamma \in \Gamma$, one has $A\cap \bigcap_{\gamma\in \Gamma} B(x_\gamma,r_\gamma) \neq \emptyset$.

\begin{Lem}\label{Lem:injective-balls}
Let $(X,d)$ be a metric space and $x_0 \in X$. Then, the following are equivalent
\begin{enumerate}[(i)]
\item $(X,d)$ is injective,
\item $B(x_0,R)$ is injective for every $R \in [0,\infty)$.
\end{enumerate}
\end{Lem}

\begin{proof}
It is clear that $(i)$ implies $(ii)$ since $X$ is hyperconvex. To prove the other implication, let $\{ (x_\gamma,r_\gamma) \}_{\gamma\in \Gamma} \subset X \times [0,\infty) $ be such that $d(x_\beta,x_\gamma) \leq r_\beta + r_\gamma$. Fix $\alpha \in \Gamma$ and for every $\beta, \gamma \in \Gamma$, define
\begin{align*}
A_\gamma &:= B(x_\gamma,r_\gamma) \cap B(x_\alpha,r_\alpha), \\
R_{\beta\gamma} &:= \max \{d(x_\alpha,x_\beta),d(x_\alpha,x_\gamma)\} + d(x_\alpha,x_0) + r_\alpha.
\end{align*}
One has $x_\beta,x_\gamma \in B(x_0,R_{\beta \gamma})$ and this last set is hyperconvex by assumption, it follows that 
\begin{enumerate}[(a)]
\item $B(x_\alpha,r_\alpha)$ is hyperconvex (since $B(x_\alpha,r_\alpha) \subset B(x_0,R_{\beta \gamma})$),
\item $A_\gamma$ is externally hyperconvex in $B(x_\alpha,r_\alpha)$ for each $\gamma \in \Gamma$,
\item $A_\beta \cap A_\gamma \neq \emptyset$ for every $\beta,\gamma \in I$.
\end{enumerate}
Hence, it follows by \cite[Proposition 1.2]{Mie} that $\bigcap_{\gamma \in \Gamma} B(x_\gamma,r_\gamma) = \bigcap_{\gamma \in \Gamma} A_\gamma \neq \emptyset$.
\end{proof}

For the characterization of Theorem~\ref{Thm:ISLI-CharacterizationTheorem} to hold, we may (as stated there) need to drop an arbitrary number of the inequalities. In order to treat all cases in a uniform way we do the following. Set $\olReal:=\{-\infty\}\cup\R\cup\{\infty\}$ and endow it with the obvious total order (also, for example $\max\{x,-\infty\} = x$ for all $x\in\R$). Then we allow the $1$-Lipschitz bounds $\underline{r\!}\,_i,\bar r_i$ to take values in $\olReal$. And a map $l_\infty(I\setminus\{i\})\to\olReal$ is called {\em $1$-Lipschitz} if it is either constant (allowing the image to be $\{-\infty\}$ or $\{\infty\}$) or it is real valued and $1$-Lipschitz in the usual sense. Now instead of dropping, say, a lower bound $\underline{r\!}\,_i$, we just set $\underline{r\!}\,_i=-\infty$ and the inequality $(\underline{r\!}\,_i \circ \widehat{\pi}_i)(x)  \le x_i$ means that no condition is imposed. The next proposition proves one of the two implications in Theorem~\ref{Thm:ISLI-CharacterizationTheorem}.

\begin{Prop}\label{Prop:p1}
For every $i\in I$ let $\underline{r\!}\,_i, \bar r_i \colon l_\infty(I\setminus\{i\})\to\olReal$ be a pair of $1$-Lipschitz functions such that $\underline{r\!}\,_i\leq\bar r_i$. Define
\begin{equation*}
Q := \left \{ x \in l_\infty(I)  : (\underline{r}_i \circ \widehat{\pi}_i)(x)  \le x_i \le (\overline{r}_i \circ \widehat{\pi}_i)(x) \text{ for all } i\in I \right \}
\end{equation*}
and assume that this set is non-empty. Then $Q$ is injective.
\end{Prop}

\begin{proof}
Let $\mathfrak{R}$ denote the set $\{\underline{r}_{i} : i \in I\} \cup \{ \overline{r}_{i} : i \in I \}$. Let $\pi_i \colon l_{\infty}(I) \to \R$ be the $i$-th coordinate projection. We divide the proof into three steps. \\

\noindent \textbf{First Step}: We first show the statement in the case $\mathfrak{R}$ is a set of $\lambda$-Lipschitz functions for some $\lambda \in [0,1)$ and $F$ is a finite subset of $I$ such that $\underline{r}_{i}=-\infty, \bar{r}_{i}=\infty$ for all $i\in I\setminus F$.  Thus we only have a finite number of non-trivial inequalities.
Assume without loss of generality that $F = \{1,\dots,N\}$.
For $i \in F$ and any $x \in l_\infty(I)$, let us define $\rho_i \in \mathrm{Lip}_{1}(l_\infty(I), l_\infty(I))$ implicitly through
\[
\pi_j \circ \rho_i(x) = 
\begin{cases}
\min \bigl \{ \max \{ x_i, (\underline{r}_i \circ \widehat{\pi}_i)(x) \} , (\overline{r}_i \circ \widehat{\pi}_i)(x) \bigr \} &\text{if }  j=i , \\
x_j &\text{otherwise}.
\end{cases} 
\]
Moreover, set $G_0 := \mathrm{id}_{l_\infty(I)}$, as well as
\[
G_i := \rho_i \circ \cdots \circ \rho_{1}
\]
and 
\[
T := G_N = \rho_{N} \circ \cdots \circ \rho_{1}.
\]
Fix now $x \in l_\infty(I)$. We show that $(T^m(x))_{m \in \mathbb{N}}$ converges to a fixed point of $T$. Let us define the maps $\{f_i\}_{i \in F} \subset \mathrm{Lip}_{\lambda}(l_\infty(I),\R)$ by 
\[
f_i \ : \ y \mapsto \min \bigl \{ \max \{\alpha_i, (\underline{r}_i \circ \widehat{\pi}_i)(y)  \}, (\overline{r}_i\circ \widehat{\pi}_i)(y) \bigr \},
\]
where $\alpha_i:=(\pi_i \circ G_{i-1} \circ T^m)(x) = (\pi_i \circ G_i \circ T^{m-1})(x)$. We further set 
\[
\beta_i := \left|\pi_i \Bigl( ( G_i \circ T^m)(x) - T^{m}(x) \Bigr) \right|
\]
for any $i \in F$ and observe that 
\begin{align*}
\beta_i 
&= \left|\pi_i \Bigl((G_i \circ T^{m})(x) - (G_i \circ T^{m-1})(x) \Bigr) \right| \\
&= \left|\pi_i \Bigl((G_i \circ T^{m})(x) - (\rho_i \circ G_i \circ T^{m-1})(x) \Bigr) \right| \\
&= \left|(f_i \circ G_{i-1} \circ T^m)(x) - (f_i \circ G_i \circ T^{m-1})(x) \right| \\
&\le \lambda \left\| ( G_{i-1} \circ T^{m})(x) - ( G_i \circ T^{m-1})(x) \right\| \\
&\le \lambda \left\|  ( G_{i-1} \circ T^m)(x) - (G_{i-1} \circ T^{m-1})(x)  \right\| \\
&\le \lambda \left\|  T^m(x) - T^{m-1}(x)  \right\|.
\end{align*}
Thus
\[
 \left\|  T^{m+1}(x) - T^{m}(x)  \right\| \le \max_{i \in F} \beta_i \le \lambda \left\|  T^{m}(x) - T^{m-1}(x)  \right\|.
\]
It easily follows that $(T^m(x))_{m \in \mathbb{N}}$ is a Cauchy sequence and thus converging to a fixed point $x^*$ of $T$. This implies in particular that $x^* \in Q$. Hence, we can define the $1$-Lipschitz retraction $\rho \colon l_\infty(I) \to Q$ to be the pointwise limit of the sequence $(T^{m})_{m \in \mathbb{N}}$. It follows that $Q$ is injective.\\

\noindent \textbf{Second Step}: We now prove the statement in case the functions in $\mathfrak{R}$ are only assumed to be $1$-Lipschitz but keeping the assumption about the finite subset $F\subset I$. Moreover, we assume without loss of generality that $0 \in Q$. By Lemma~\ref{Lem:injective-balls}, it is enough to show that for any $R > 0$, the set $Q \cap B(0,R)$ is injective. For each $i \in I$, we set
\begin{align*}
(\underline{s}_i \circ \widehat{\pi}_i)(x) &:= \min \{ \max \{ (\underline{r}_i \circ \widehat{\pi}_i)(x),-R\},R\}\\
(\overline{s}_i \circ \widehat{\pi}_i)(x) &:= \min \{ \max \{ (\overline{r}_i \circ \widehat{\pi}_i)(x),-R\},R\}.
\end{align*}
Using $0 \in Q$, a short calculation yields that $\underline{r}_i \le \overline{r}_{i}$ implies $-R \le \underline{s}_i \le \overline{s}_{i} \le R$. Set
\[
P := \Bigl \{ x \in B(0,R) : (\underline{s}_i \circ \widehat{\pi}_i)(x) \le x_i \le (\overline{s}_i \circ \widehat{\pi}_i)(x) \text{ for all } i \in F \Bigr \}.
\]
Since the functions $\underline{s}_i$ and $\overline{s}_i$ are $1$-Lipschitz and using that $0 \in Q$ again, one has
\[
P = Q \cap B(0,R).
\]
We can thus, for $k \in \mathbb{N}$ and $i \in F$, set $\lambda_k := 1 - \frac{1}{k}$ and define $\underline{s}_i^k,\overline{s}_i^k$ through
\begin{align*}
(\underline{s}_i^k \circ \widehat{\pi}_i)(x) &= \lambda_k \bigl[ (\underline{s}_i \circ \widehat{\pi}_i)(x) + R \bigr] - R, \\
(\overline{s}_i^k \circ \widehat{\pi}_i)(x) &= \lambda_k \bigl[ (\overline{s}_i \circ \widehat{\pi}_i)(x) - R \bigr] + R.
\end{align*}
Note that 
\[
-R \le \underline{s}^k_i \le \underline{s}_i \le \overline{s}_{i}  \le \overline{s}^k_{i} \le R.
\]
For any $k \in \mathbb{N}$, we now set
\[
Q_k := \Bigl \{ x \in B(0,R) : (\underline{s}_i^k \circ \widehat{\pi}_i)(x) \le x_i \le (\overline{s}_i^k \circ \widehat{\pi}_i)(x) \text{ for all } i \in F \Bigr \}.
\]
The functions in $\mathfrak{R}^k := \{\underline{s}_i^k : i \in F \} \cup \{\overline{s}_{i}^k : i \in F \}$ are all $\lambda_k$-Lipschitz. Hence, we can apply the first step and define the $1$-Lipschitz retraction $\rho^k \colon B(0,R) \to Q_k$ to be the pointwise limit of the sequence $(T^{m,k})_{m \in \mathbb{N}}$. Since $B(0,R)$ is injective, it follows that $Q_k $ is injective. Finally, since the sequence $(Q_k)_{k \in \mathbb{N}}$ is decreasing for the inclusion and
\[
\bigcap_{k \in \mathbb{N}} Q_k = P = Q \cap B(0,R) ,
\]
it follows that $Q \cap B(0,R)$ is injective by Theorem~\ref{Thm:IS-TheoremBaillon}. So  
\[
Q=\left \{ x \in l_\infty(I) : (\underline{r}_i \circ \widehat{\pi}_i)(x)  \le x_i \le (\overline{r}_i \circ \widehat{\pi}_i)(x) \text{ for all } i\in F \right \}
\]
is injective by Lemma~\ref{Lem:injective-balls}.\\

\noindent \textbf{Third Step}: Let $\mathcal{F}$ be the family of all finite subsets of $I$. For every $F \in \mathcal{F}$, let
\[
Q^F := \left \{ x \in l_\infty(I) : (\underline{r}_i \circ \widehat{\pi}_i)(x)  \le x_i \le (\overline{r}_i \circ \widehat{\pi}_i)(x) \text{ for all } i\in F \right \}.
\]
As it is shown just above, $Q^F$ is injective. Therefore, for every $R \in [0,\infty)$, the set $A^F := Q^F \cap B(0,R)$ is injective by Lemma~\ref{Lem:injective-balls}. Let 
\[
\mathcal{M} := \left \{ J \subset I : A^{J\cup F} \text{ is injective for all $F\in\mathcal{F}$} \right \}
\]
be partially ordered by inclusion. By the second step, $\emptyset \in \mathcal{M}$. Moreover, if $F \in \mathcal{F}$ and if $(J_{\gamma})_{\gamma \in \Gamma}$ is a chain in $\mathcal{M}$, we can set $J_{\Gamma} := \bigcup_{\gamma \in \Gamma} J_{\gamma}$ and we obtain by Theorem~\ref{Thm:IS-TheoremBaillon} that $A^{J_\Gamma\cup F} = \bigcap_{\gamma \in \Gamma} A^{J_{\gamma} \cup F}$ which is injective by Theorem~\ref{Thm:IS-TheoremBaillon}. We can thus use Zorn's lemma to deduce the existence of a maximal element $M \in \mathcal{M}$. By maximality, it follows that $M = I$, which implies that the set  
\[
 A^I = Q \cap B(0,R)
\]
is injective as well. Again by Lemma~\ref{Lem:injective-balls}, it follows that $Q$ is injective and this concludes the proof.
\end{proof}

\begin{Rem}\label{Rem:ISLI-ExtremalFunctionsDistances}
Let $X$ and $Y$ be metric spaces, and let $i \colon X \to Y$ be an isometric embedding. As stated in \cite[(3)~in~Proposition~3.4]{Lan}, the following are equivalent:
\begin{enumerate}[(i)]
\item $(i,Y)$ is an injective hull of $X$.
\item $(i,Y)$ is a minimal injective extension of $X$, that is, $Y$ is injective and no proper subspace of $Y$ containing $i(X)$ is injective.
\end{enumerate}
It follows that if $X$ is an injective metric space, then $\e(X) = \E X$.  In our case it follows that if $Q \subset l_{\infty}(I)$ is injective, then the distance functions $\{ d_q \}_{q \in Q}$ are the only extremal functions in $\Delta(Q)$. Therefore, given any point $x \in l_{\infty}(I) \setminus Q$, the function $d_x \in \R^Q$ given by the assignement $q \mapsto \| x - q \|$ verifies $d_x(Q) \in (0,\infty)$, thus $d_x \notin \e(Q) = \E Q $, in other words $d_x \in \Delta(Q)$ is not extremal.
\end{Rem}

In the context of subsets of $l_{\infty}(I)$, we need to introduce some pieces of notation. Recall that $\pi_i \colon l_{\infty}(I) \to \R$ is the $i$-th coordinate projection. For $i \in I$, we set 
\[
C_i := \{ x \in  l_{\infty}(I) : x_i = \| x \| \}.
\]
For $S \subset  l_{\infty}(I)$, we define $-S := \{ -s : s \in S\}$ and we write $p + S$ for the set $\{ p + s : s \in S \}$, i.e.~the translate of $S$ by $p$. Note that the interior of $C_i$ satisfies
\[
\mathrm{Interior}(C_i) := \Biggl \{ x \in  l_{\infty}(I) : x_i > \sup_{j \in I \setminus \{ i \}} | x_j | \Biggr \}.
\]
The next proposition is the last piece needed to prove Theorem~\ref{Thm:ISLI-CharacterizationTheorem}. 

\begin{Prop}\label{Prop:p2}
If $Q \subset l_{\infty}(I)$ is injective, then $Q$ satisfies the assumptions of Proposition~\ref{Prop:p1}.
\end{Prop}

\begin{proof}
The injective subsets of $\R$ are exactly the closed intervals and the result clearly holds in this case. Therefore, assume that $|I| \ge 2$. By Remark~\ref{Rem:ISLI-ExtremalFunctionsDistances}, we may assign to every $x \in l_{\infty}(I) \setminus Q$ the positive quantity
\begin{align*}
&\eps (x) := \\
&\sup\{\eps \in \R \,|\, \text{there is }p\in Q \text{ with } \|x-p\|+\|x-q\|\geq \|p-q\|+\eps \text{ for all } q\in Q \}.
\end{align*}
Choosing $q\in Q$ such that $\|x-q\|=d(x,Q)$, one can easily see that $\eps(x)\leq 2 d(x,Q)$. For every $x \in l_{\infty}(I) \setminus Q$, let $p_x \in Q$ be such that for every $q\in Q$, one has 
\[
\inf_{q \in Q} \Bigl(\|x-p_x\|+\|x-q\| - \|p_x-q\| \Bigr) \ge \tfrac{\eps(x)}{2}.
\]
Next we select a cone $C_x$ for every $x\in l_\infty^n\setminus Q$. To that end, let $\alpha$ be some positive real parameter. We determine the value of $\alpha$ in the course of the proof. For any $\delta \in (0, \tfrac{1}{2} \|x-p_x\|)$, one can find $i \in I$ such that 
\[
|\pi_i(x - p_x)|  \ge \|x-p_x\| - \delta > 0.
\]
($\delta$ is the additional parameter needed to generalize from finite to infinite index sets~$I$.)
Now let $e^i \in l_{\infty}(I)$ be given by $e^i_j = \delta_{ij}$ i.e. $e^i$ is everywhere equal to zero except at $i$ where it is equal to one. We set  
\[
C_x := 
\begin{cases}
x-\alpha\eps(x)e^i +C_i &\text{if }  \pi_i(x - p_x) > 0 , \\
x+\alpha\eps(x)e^i - C_i &\text{if }  \pi_i(x - p_x) < 0.
\end{cases}
\]
Observe that $x\in\operatorname{Interior}(C_x)$ holds for every $x$. Assume that we are in the case $C_x:=x-\alpha\eps(x)e^i + C_i$ (the case $C_x:=x+\alpha\eps(x)e^i-C_i$ is analogue). Assume that $Q\cap C_x \neq \emptyset$ and pick $q$ in this intersection, one then has 
\[
w:= q + \alpha \eps(x)e^i \in C_x + \alpha\eps(x)e^i = x + C_i.
\]
Therefore $\|w - x\| = \pi_i(w - x)$ and thus
\begin{align*}
\|w - p_x\| \ge | \pi_i(w - p_x) | 
= |\pi_i(w - x + x - p_x) |  
&= |\pi_i(w - x)| + |\pi_i(x - p_x) | \\
&\ge \|w - x\| + \|x - p_x \| - \delta.
\end{align*}
Consequently
\[
\|p_x-q\|\geq \|p_x-x\|+\|x-q\|-2\alpha\eps(x)-\delta.
\]
Choosing $\alpha < \tfrac{1}{8}$ and $\delta < \tfrac{\eps(x)}{4}$ we obtain a contradiction to the definition of $p_x$. Hence, we do so and obtain $Q\cap C_x = \emptyset$ for all $x\in l_{\infty}(I) \setminus Q$. 

For every $i$, the function $\bar{r}_i$ is defined to be the pointwise infimum over the family of $1$-Lipschitz functions $l_{\infty}(I \setminus \{ i \} ) \to \R$ defined by the assignement 
\[
y \mapsto \|\widehat{\pi}_i(x)-y\|+\pi_i(x)-\alpha\eps(x)
\]
where every $x$ such that $C_x=x-\alpha\eps(x)e^i+C_i$ contributes exactly one member. If there is no such $x$, we let $\bar{r}_i:=\infty$. Similarly, $\underline{r\!}\,_i:=-\infty$ if there is no $x$ with $C_x=x+\alpha\eps(x)e^i - C_i$ or otherwise the supremum over all functions
\[
y\mapsto\|\widehat{\pi}_i(x)-y\|+\pi_i(x)+\alpha\eps(x)
\]
for $x$ with $C_x=x+\alpha\eps(x)e^i - C_i$. It is not difficult to deduce from $x\in\operatorname{Interior}(C_x)$ and $Q\cap C_x = \emptyset$ that 
\[
Q = \{ x \in l_{\infty}(I) : \underline{r\!}\,_i(\widehat{\pi}_i(x)) \leq x_i
\leq
\bar{r}_i(\widehat{\pi}_i(x)) \text{ for } i \in I \}.
\]

It remains to be shown that the inequalities $\underline{r\!}\,_i\leq\bar{r}_i$ hold. First, note that if $\underline{r}_i(p)>\bar{r}_i(p)$ at some $p \in l_{\infty}(I \setminus \{ i \} )$, then there are points $x,y \in l_{\infty}(I) \setminus Q$ with $C_x:=x-\alpha\eps(x)e^i + C_i$ and $C_y:=y+\alpha\eps(y)e^i - C_i$ such that the intersection $\operatorname{Interior}(C_x)\cap\operatorname{Interior}(C_y)$ is non-empty. To show that this can not happen for appropriate choice of $\alpha$, we assume that $\operatorname{Interior}(C_x)\cap\operatorname{Interior}(C_y) \neq \emptyset$ and start by noting that the apex $x-\alpha\eps(x)e^i$ of $C_x$ lies in $\operatorname{Interior}(C_y)$. Therefore $x':=x-\alpha\eps(x)e^i-\alpha\eps(y)e^i$ lies in $\operatorname{Interior}(y - C_i)$ and the same holds for $p_x':=p_x-\alpha\eps(x)e^i-\alpha\eps(y)e^i$ since $p_x\in x - C_i$.

\noindent So we have 
\begin{align*}
\|x-p_x\|+\|x-p_y\|&\leq \|x'-p_x'\|+\|x'-p_y\|+\alpha(\eps(x)+\eps(y)) \\
&= \|p_x'-p_y\|+\alpha(\eps(x)+\eps(y)) \\
& \leq \|p_x-p_y\|+2\alpha(\eps(x)+\eps(y)),
\end{align*}
hence by definition of $p_x$ and $p_y$ this leads to $\eps(x)\leq 4\alpha(\eps(x)+\eps(y))$ and $\eps(y)\leq 4\alpha(\eps(x)+\eps(y))$, respectively. Now take $\alpha<\tfrac{1}{8}$. The sum of the last two inequalities involving $\eps(x)$ and $\eps(y)$ then yields a contradiction. Therefore, $\operatorname{Interior}(C_x)\cap\operatorname{Interior}(C_y)$ is empty and this finishes the proof.
\end{proof}

\section{Examples}

We start with a remark.

\begin{Rem}
Any codimension one linear subspace $V$ of $l_{\infty}(I)$ is injective if and only if there is an $i \in I$ such that 
\begin{equation}\label{eq:ISLI-EmptyConesCodimensionOne}
V \subset l_{\infty}(I) \setminus ( \mathrm{Interior}(-C_i) \cup \mathrm{Interior}(C_i) ).
\end{equation}
We first show that if $V$ is injective, there exists a coordinate $i$ as in \eqref{eq:ISLI-EmptyConesCodimensionOne}. Assume that the converse holds, namely that for every $j \in I$, one can pick
\[
v^j \in V \cap  ( \mathrm{Interior}(-C_j) \cup \mathrm{Interior}(C_j) ),
\]
Choose arbitrarily $p \in l_{\infty}(I)$ and assume without loss of generality that $v^j \in V \cap \mathrm{Interior}(C_j)$. Note that for any $\alpha \in [0,\infty)$, one has $\alpha v^j \in V \cap \mathrm{Interior}(C_j)$. Let $A(v^j):= |\pi_j(v^j)|$ and $B(v^j):= \sup_{i \in I \setminus \{ j \}} |\pi_i(v^j)|$. Clearly, there is then an $\eps \in [0,\infty)$ such that $A(v^j) \ge (1+\eps) B(v^j)$. A short calculation shows that putting $\alpha_j := \tfrac{2 \| p \|}{\eps B(v^j)}$, one obtains $A(\alpha_j v^j) - B(\alpha_j v^j) \ge 2 \| p \|$. By choice of $\alpha_j$, it then follows that
\begin{align*}
|\pi_j(p_j - \alpha_j v^j)| &= \| p_j - \alpha_j v^j \| \\
|\pi_j(p_j + \alpha_j v^j)|&= \| p_j + \alpha_j v^j \|.
\end{align*}
Therefore, one infers that 
\[
\{ p \} = \bigcap_{j \in I} \Bigl( B( \alpha_j v^j, \| p - \alpha_j v^j \|) \cap B( -\alpha_j v^j, \| p + \alpha_j v^j \|) \Bigr).
\]
Picking $p \notin V$, it follows that $V$ is not hyperconvex and thus not injective. This shows that if $V$ is injective, such a coordinate as in \eqref{eq:ISLI-EmptyConesCodimensionOne} exists. Conversely, if such a coordinate $i$ exists, then $V$ can be expressed as
\begin{equation}\label{eq:ISLI-ExpresionConesV}
V = \Bigl \{ x \in l_\infty(I) : (\underline{r}_i \circ \widehat{\pi}_i)(x)  \le x_i \le (\overline{r}_i \circ \widehat{\pi}_i)(x) \Bigr \}
\end{equation}
where 
\[
(\underline{r}_i \circ \widehat{\pi}_i)(x) := \sup_{ y \in l_{\infty}(I)  \setminus V} (y_i - \| \widehat{\pi}_i(x) - \widehat{\pi}_i(y) \|) 
\]
and
\[
(\overline{r}_i \circ \widehat{\pi}_i)(x) := \inf_{ y \in l_{\infty}(I)  \setminus V} (y_i + \| \widehat{\pi}_i(x) - \widehat{\pi}_i(y) \|) .
\]
It is easy to see by \eqref{eq:ISLI-EmptyConesCodimensionOne} that $V$ is a subset of the right-hand side of \eqref{eq:ISLI-ExpresionConesV}. Now, we prove that the complement of $V$ is contained in the complement of the right-hand side of \eqref{eq:ISLI-ExpresionConesV}. Consider the map $\pi^i_V \colon l_\infty(I) \to V$ which corresponds to the projection onto $V$ along the $i$-th coordinate which is a well-defined map since $i$ satisfies \eqref{eq:ISLI-ExpresionConesV}. For any $y \notin V$, one either has $\pi_i(y) > \pi_i(\pi^i_V(y))$ which implies $\pi_i(y) > (\overline{r}_i \circ \widehat{\pi}_i)(y)$ or $\pi_i(y) < \pi_i(\pi^i_V(y))$ which implies $\pi_i(y) < (\underline{r}_i \circ \widehat{\pi}_i)(y)$. This proves the desired inclusion and thus that \eqref{eq:ISLI-ExpresionConesV} holds. By Proposition~\ref{Prop:p1}, it follows that $V$ is injective and this finishes the proof of the equivalence.
\end{Rem}

\begin{Expl}
In case $I:=\mathbb{N}$ we consider the set $V:= \mathrm{ker}(\Lambda)$ where $\Lambda \colon l_{\infty}(\mathbb{N}) \to \R$ denotes a real Banach limit, namely $\Lambda \in ( l_{\infty}(\mathbb{N}))^*$ satisfies the following propeties
\begin{enumerate}
\item Let $x := (x_n)_{n \in \mathbb{N}}$ be a sequence with non-negative terms, then $\Lambda(x) \ge 0$. 
\item If $S \colon  l_{\infty}(\mathbb{N}) \to  l_{\infty}(\mathbb{N})$ denotes the left-shift operator given by the relation $\pi_n \circ S = \pi_{n+1}$, one has $\Lambda \circ S = \Lambda$.
\item For every convergent sequence $x := (x_n)_{n \in \mathbb{N}}$, one has $\Lambda(x) = \lim_{n \rightarrow \infty} x_n$.
\end{enumerate}
One can see that by invariance of the  Banach limit under left-shift, $V$ contains all sequences having only finitely many non-zero entries. It is then easy to see that there is no $i \in \mathbb{N}$ satisfying \eqref{eq:ISLI-EmptyConesCodimensionOne} and thus $V$ is not injective.
\end{Expl}

\begin{Expl}
In case $I:=\mathbb{N}$ we can also consider the set $V:= \mathrm{ker}(\Phi)$ where $\Phi \colon l_{\infty}(\mathbb{N}) \to \R$ denotes an element in $\alpha(l_1(\mathbb{N}))$ with $\alpha \colon l_1(\mathbb{N}) \to (l_{\infty}(\mathbb{N}))^{*}$ standing for the canonical isometric embedding induced by the one of $l_1(\mathbb{N})$ into its double dual. It is then easy to see that \eqref{eq:ISLI-EmptyConesCodimensionOne} holds if and only if $\| \alpha^{-1}(\Phi) \|_1 \le 2 \| \alpha^{-1}(\Phi) \|_{\infty}$, compare to \cite{Moe,Pav}.
\end{Expl}



\end{document}